\documentclass[11pt,reqno]{amsart}
\usepackage{amscd,amsmath,amssymb,latexsym,amsthm,amsfonts}
\usepackage[utf8]{inputenc}
\usepackage{enumerate,color}
\usepackage{tikz-cd}
\tikzset{smalltext/.style={"\textup{\small #1}" description}}

\DeclareMathAlphabet{\bsf}{OT1}{cmss}{bx}{n}

\usepackage{graphics}
\usepackage[all]{xy}

\def\qtq#1{\quad\text{#1}\quad}

\def\setof#1#2{\{#1 \, |\, #2\} }

\def\sue{\subseteq}

\def\2{\text{\bf 2}}

\def\of{\mathfrak o}
\def\cf{\mathfrak c}

\def\OS{\text{\sf O}}

\def\loc{\text{{\bf Loc}}}

\newcommand{\tbigwedge}{\mathop{\textstyle \bigwedge}}
\newcommand{\tbigcap}{\mathop{\textstyle \bigcap}}
\newcommand{\tbigcup}{\mathop{\textstyle \bigcup}}
\newcommand{\tbigvee}{\mathop{\textstyle \bigvee}}

\def\SL{\mathsf{S}}

\def\mo{\mathcal{O}}
\def\mc{\mathcal{C}}

\newtheorem{theo}{Theorem}[section]
\newtheorem{coro}[theo]{Corollary}
\newtheorem{lemm}[theo]{Lemma}
\newtheorem{prop}[theo]{Proposition}

\theoremstyle{definition}

\newtheorem{rema}[theo]{Remark}
\newtheorem{remas}[theo]{Remarks}
\newtheorem{defin}[theo]{Definition}

\begin{document}


\title{Fitness and subfitness via interior and closure operators}

\author[João Areias]{João Areias}

\address{\hspace*{-\parindent}Department of Mathematics, University of Coimbra,  3000-143 Coimbra,\newline Portugal \newline {\it Email address}: {\tt uc2020217955@student.uc.pt}}



\begin{abstract}
    {  We start by rephrasing some well-known characterizations of fit and subfit locales in terms of closure and interior operators on the system of sublocales. Using these characterizations, we are able to extend the notions of fitness and subfitness to a more general setting. This approach not only allows us to obtain more general results, but also to present new properties of fitness and subfitness (such as a formula for the computation of supplements of sublocales). Using this approach we were also able to simplify the proofs of some well-known results.}
\end{abstract}
\maketitle
\section{Introduction}
Subfitness and fitness are two separation conditions that play a crucial role in point-free topology. The main goal of this paper is to present a new approach to these topics, which will enable us to obtain more general results. This approach will also allow us to present shorter and more direct proofs of some well-known results than those typically found in the literature. In fact, the majority of our proofs are nothing more than direct computations.

We start by rephrasing some well-known characterizations of subfitness and fitness of a locale $L$ in terms of closure and interior operators on the coframe $\SL(L)$ of the sublocales of $L$.

These characterizations motivate the introduction of the notions of lifted-subfit and lifted-fit coframes. These are the coframes that behave similarly to $\SL(L)$ when $L$ is a subfit or a fit locale, respectively. We are then able to extend some well-known results from locales to this more general setting.

As a consequence of the general theory we have developed, we derive some new results about separation on locales. The most relevant of them is Proposition \ref{prop sup em sl}, where we give a formula for the supplement of sublocales in terms of their associated nuclei.

Our approach may also have some interest from a logical standpoint because, unlike the proofs usually found in the literature,  we never need to proceed by contradiction.

This paper is structured as follows. In Section \ref{preliminares} we recall some facts we will later need, and characterize  subfit and fit frames in terms of closure and interior operators on $\SL(L)$. In Section \ref{sec oclo cint} we establish  the setting in which we will work for the remainder of the paper,  as well as some basic properties of our operators. In Section \ref{sec lifted subfit} we introduce the notion of lifted-subfit coframe. In Section \ref{sec consequencias lifte subfit} we study the consequences of lifted-subfitness, namely a formula for the complements in $\SL(L)$ and a generalization of Joyal-Tierney's Theorem. In Section \ref{sec lifted fit} we introduce the notion of lifted-fit coframe. In Section \ref{sec heredi} we discuss how lifted-subfit and lifted-fit coframes may be characterized by the behavior of their principal down-sets. We finish by studying the separation conditions on zero-dimensional coframes.

\section{Preliminaries}\label{preliminares}
\subsection{Frames and Locales}
In this section, we recall some relevant notions and facts about frames and locales that we will need. They are all taken from \cite{PP}.

      A \emph{frame} (also known as a \emph{locale}) is a complete lattice $L$ satisfying the infinite distributivity law
    $$a\wedge \tbigvee B=\tbigvee \{a\wedge b\mid b\in B\}$$
    for all $a\in L$ and $B\subseteq L$. The dual of a frame is a {\em coframe} that is, a complete lattice $K$ satisfying
    $$a\vee \tbigwedge B=\tbigwedge \{a\vee b\mid b\in B\}$$
    for all $a\in K$ and $B\subseteq K$.

Every frame  has a {\em Heyting structure} 
with the Heyting operation $\to$ satisfying
\[
a\wedge b\leq c \qtq{iff} a\leq b\to c
\]
where $b\to c=\tbigvee\{a\in L\mid a\wedge b\leq c\}$.

\subsection{Sublocales}

    A {\em sublocale} of a locale $L$ is a subset $S\sue L$ that satisfies the following conditions:
\begin{enumerate}[(Sl1)]
\item
$S$ is a {\em meet-subset}, that is, for every $M \subseteq S$, $\tbigwedge M\in S$.
\item
For every $s\in S$ and every $x\in L$, $x \to s\in S$.
\end{enumerate}

The system $\SL(L)$ of all sublocales of $L$ (partially ordered by inclusion) is a coframe  with a fairly transparent structure (see \cite{PP}, pp. 28-29):

\begin{equation*}\label{joinsub}
\tbigwedge_i S_i
= \tbigcap_i S_i
\qtq{and} \tbigvee_i S_i
= \setof{\tbigwedge M}{M\subseteq \tbigcup_i S_i}.
\end{equation*}

The bottom element of $\SL(L)$ is the subset of $L$ that contains only the top element $1$ and is denoted by $\OS$.

Given a locale $L$, one can associate each of its elements $a\in L$ with the following sublocales of $L$:
\[
\of a=\setof {a\to x}{x\in L} \qtq{and} \cf a=\{x\in L\mid x\ge a\}.
\]

These sublocales  are called \emph{open sublocales} and \emph{closed sublocales} respectively. The system of all open sublocales of $L$ is denoted by $\of L$, and the system of all closed sublocales of $L$ is denoted by $\cf L$.
Moreover,
\[
\begin{aligned}
&\of(0)=\OS,\; \of(1)=L,\ \ \of a\cap\of b=\of(a\wedge b) \ \text{ and }\ \tbigvee_i\of a_i=\of(\tbigvee_i a_i),\\
&\cf(1)=\OS,\; \cf(0)=L,\ \ \cf a\vee\cf b=\cf(a\wedge b) \ \text{ and }\ \tbigcap_i\cf a_i=\cf(\tbigvee_i a_i).
\end{aligned}
\]

For each $S\in \SL (L)$ one defines
\begin{itemize}
    \item the interior of $S$: $S^\circ=\tbigvee \{\of a \in \of L\mid \of a\sue S\}$;
    \item the closure of $S$: $S^-=\tbigcap \{\cf a \in \cf L \mid S\sue \cf a\}$.
\end{itemize}

\subsection{Localic maps}

    A map $h\colon M\to L$ between frames is a \emph{frame homomorphism} if it has a right adjoint and preserves  finite meets.
    A map $f\colon L\to M$ between locales is a \emph{localic map} if it has a left adjoint $f^*$ that preserves finite meets (thar is, localic maps are the right adjoints of frame homomorphisms).

The category of locales and localic maps is denoted by $\loc$.

Let $f\colon L\to M$ be a localic map, and $S\sue L$ and $T\sue M$ be sublocales.
Then the set-theoretic image $f[S]$ is a sublocale of $M$, whereas the set-theoretic preimage $f^{-1}[T]$ is not, in general, a sublocale of $L$. Nevertheless, there exists the  largest sublocale contained in $f^{-1}[T]$, called the {\em localic preimage}
 of $T$ and denoted  by $f_{-1}[T].$
 
 There is the obvious   adjunction
\[
f[S]\sue T \Leftrightarrow S\sue f_{-1}[T].
\]
Moreover
$$f_{-1}[\of b]=\of (f^\ast (b)) \qtq{ and }f_{-1}[\cf b]=\cf (f^\ast (b)).$$

 For more information on frames and locales, we refer to \cite{PP}.

\subsection{Supplements}
In this section we recall some basic facts about supplements. All the results are taken from \cite{remainders}.

    Given a coframe $K$ and $a\in K$, the \emph{supplement} of $a$ is the element
    $$a^\#=\tbigwedge \{y\in K \mid a\vee y= 1\}.$$

The supplement satisfies the following properties:
    \begin{enumerate}
        \item $a^\#\le b \Leftrightarrow a\vee b=1\Leftrightarrow b^\#\le a$.
        \item $a\le b \Rightarrow b^\#\le a^\#$.
        \item $a^{\#\#}\le a$.
        \item $a^\#=a^{\#\#\#}$.
        \item $0^\#=1$ and $1^\#=0$.
        \item $(\tbigwedge a_i)^\#=\tbigvee a_i^\#$.
    \end{enumerate}

Of course $a \vee a^\#=1$, while $a\wedge a^\#\ge 0$, in general. Whenever $a\wedge a^\#=0$ it is said that $a$ is complemented. 
If $a,b\in K$ are complemented, then they also satisfy the following:
    \begin{enumerate}
        \item [(7)] $a\le b^\# \Leftrightarrow b\le a^\#.$
        \item [(8)]$a^{\#\#}= a$.
    \end{enumerate}

    If $L$ is a locale, then for every $a\in L$, the sublocales $\of a$ and $\cf a$ are complemented to each other in the coframe $\SL (L)$. Therefore $S^\circ=S^{\# -\#}$ for every $S\in\SL (L)$.

\subsection{Separation on Locales}

In this section we recall some basic facts about subfit and fit locales. All the results are taken from \cite{PPsep}.
 
    Recall that a locale $L$ is said to be \emph{subfit} whenever
    \begin{center}
        if $a\not \leq b$, then there exists $c\in L$ such that $a\vee c=1\not =b\vee c$.
    \end{center}

Subfit locales can be characterized by the properties of their sublocales in the following ways.

\begin{prop}
The following conditions on a locale $L$  are equivalent.
\begin{enumerate}
    \item $L$ is subfit.
    \item Every open sublocale of $L$ is a join of closed ones.
    \item For every sublocale $S\not = L$ there exists an $\of a$ such that
    $$S\sue \of a \subset L.$$
\end{enumerate}
\end{prop}
\medskip

    A locale $L$ is said to be \emph{fit} whenever
        \begin{center}
        if $a\not \leq b$ then there exists $c\in L$ such that $a\vee c=1$ and $ c\to b\not =b.$
    \end{center}

Fit locales can be characterized by the properties of their sublocales in the following ways.

\begin{prop}
The following conditions on a locale $L$  are equivalent.
    \begin{enumerate}
        \item $L$ is fit.
        \item Every closed sublocale of $L$ is an intersection of open ones.
        \item Every sublocale of $L$ is an intersection of open ones.
    \end{enumerate}
\end{prop}

These characterizations motivate the following definitions.

\begin{defin}\label{box and sim in loc}
    For each $S\in \SL (L)$ we define
\begin{itemize}
    \item the $\cf$-interior of $S$: $S^\Box=\tbigvee \{\cf a \in \cf L\mid \cf a\sue S\}$;
    \item the $\of$-closure of $S$: $S^\sim=\tbigcap \{\of a \in \of L \mid S\sue \of a\}.$\footnote{The $\of$-closure has also been called the \emph{the other closure}  and has been used to study the complete sublocales (see, \cite{CPP}).}
\end{itemize}
\end{defin}
It is easy to verify that the $\cf$-interior and the $\of$-closure are an interior and a closure operator on $\SL(L)$, respectively.

Using these operators, we can rewrite the previous characterizations of subfit and fit locales in the following ways:

\begin{coro}\label{coro subfit locales}
The following conditions on a locale $L$  are equivalent :
    \begin{enumerate}
    \item $L$ is subfit.
    \item For every $a\in L$, $\of a^\Box=\of a$. 
    \item For every $S\in \SL (L)$, $S^\sim =L$ if and only if $S=L.$
\end{enumerate}
\end{coro}

\begin{coro}\label{coro fit locales}
    The following conditions on a locale $L$  are equivalent:
    \begin{enumerate}
        \item $L$ is fit.
        \item For every $a\in L$, $\cf a^\sim=\cf a$. 
        \item For every $S\in \SL(L)$, $S^\sim =S.$
    \end{enumerate}
\end{coro}

The main goal of this paper is to show that, by characterizing subfitness and fitness by way of the $\cf$-interior and $\of$-closure, we can not only extend a lot of results known for the case of locales to a more general setting (as specified in the following section), but also present simpler proofs of well-known results about subfit and fit locales.

\section{The setting: The $\mo$-closure and $\mc$-interior}\label{sec oclo cint}
Throughout this paper we will always consider the following setting:
\medskip
\begin{equation*}
    \begin{tabular}{|p{9cm}|}
\hline
\begin{itemize}
    \item $K$ is an arbitrary coframe.
    \item $\mo\sue K$ such that every $a\in\mathcal{O}$ has a complement  in $K$.
    \item $\mathcal{O}$ is closed under finite joins and  $1\in \mo$.
    \item $\mathcal{C}=\{a^\#\mid a\in \mathcal{O}\}$.

\end{itemize}
\\
\\
    \hline
    \end{tabular}
\end{equation*}
\medskip

We start by generalizing the notions of $\cf$-interior and $\of$-closure to this new setting.
\begin{defin}
    For each $s\in K$ we define
\begin{itemize}
    \item the $\mc$-interior of $s$: $s^\Box=\tbigvee \{c\in \mc \mid c\leq s\};$
    \item the $\mo$-closure of $s$: $s^\sim=\tbigwedge \{a\in \mo \mid s\leq a\}.$
\end{itemize}
\end{defin}

Observe that the $\mc$-interior and the $\mo$-closure are  an interior and  a closure operator on $K$ respectively.
\begin{rema}
 Note that if $L$ is a locale, $K=\SL (L)$ and $\mo=\of L$, then the $\mc$-interior and the $\mo$-closure are just the operators defined in \ref{box and sim in loc}.
    
    Whenever we are working in the coframe $\SL(L)$,  we will always consider the  $\mc$-interior and the $\mo$-closure defined in \ref{box and sim in loc}, unless explicitly said otherwise.
\end{rema}

We now present some basic properties of these operators.
\begin{prop}\label{boxsim}
For every $s\in K$,
    $s^{\#\sim\#} =s^{\Box}.$
\end{prop}
    \begin{proof}
        Notice that
    $s^{\#\sim\#}=(\tbigwedge\setof{a\in \mathcal{O}}{s^\# \leq a})^\#=\tbigvee\setof{c\in \mc}{c\leq s}=s^\square.$
    \end{proof}
    Applying the results from \cite{C} we have:
    \begin{coro}\label{sup}
 For every $s\in K$
    \begin{enumerate}
        \item $s^{\sim \# \Box}=s^{\sim \#}.$
        \item $s^\square=s^{\#\# \square}.$
        
        If, in addition, $s$ is complemented, then we have 
        
        \item $s^{\#\Box}=s^{\sim \#}.$
    \end{enumerate}
\end{coro}

\begin{prop}\label{prop sim distr}
    For every $s, t \in K$ 
    $(s\vee t)^\sim=s^\sim \vee t^\sim.$
\end{prop}
\begin{proof}
    It suffices to check that 
     $(s\vee t)^\sim\leq s^\sim \vee t^\sim.$
     
    Since $\mo$ is closed for finite joins, we have
    \begin{align*}
 s^\sim \vee t^\sim &=\tbigwedge\{a \in\mo \mid s\le a \} \vee \tbigwedge\{b \in\mo \mid t\le b \}\\
 &=\tbigwedge\{a\vee b \in\mo \mid s\le a, \;  t\le b \}\\
  &\geq\tbigwedge\{a\vee b \in\mo \mid s \vee t\le a \vee b \}=(s\vee t)^\sim.\qedhere
    \end{align*}
\end{proof}


\subsection{Associated systems}

To each of these operators, $\Box$ and $\sim$, we can associate the following subsets of $K$
$$K_\Box=\{s^\Box\mid s\in K\}\quad\mbox{ and }\quad K_\sim=\{s^\sim\mid s\in K\}.$$
 In the case where $L$ is a locale, $K=\SL(L)$ and $\mo=\of L$, it is usual to write \cite{PPsep} $K_\Box$ as $\SL_\cf(L)$ and $K_\sim$ as $ \SL_\of(L)$.

We will also consider
$K_\#=\{s^\#\mid s\in K\}=\{s^{\#\#}\mid s\in K\}.$

Whenever talking about $K_\Box$, $K_\sim$ and $K_\#$ we will always consider them to be endowed with the order inherited from $K$.

 The subset $K_\#$ is a Boolean algebra, and it is usually called the \textit{Booleanization} of $K$ (\cite{PP}, pp. 334).
 
\medskip

It is to check that:
\begin{prop}
    For every coframe $K$, $K_\sim$ is a sub-coframe of $K$.
\end{prop}

 Perhaps more surprisingly, it was proven by Picado, Pultr, Tozzi \cite{PPT} that the system $\SL_\cf (L)$ is always a frame. They did it by using an adjunction between the upsets of $L$ and the meet-sets of $L$.

Here we point out that, by adapting the arguments presented by M. Ernè in \cite{ME} we can present a direct proof that $K_\Box$ (and in particular $\SL_\cf (L)$) is a frame. Since \cite{ME} is not yet published we present the proof here.

We will need the following lemma.
\begin{lemm}[\cite{PP}, p. 106]
    Let $K$ be a coframe. For every complemented $a\in K$ and every family $\{b_i\mid i\in I\}\subseteq K$, we have
    $$a\wedge \tbigvee_{i\in I}b_i=\tbigvee_{i\in I}(a\wedge b_i) .$$
\end{lemm}

\begin{prop}
    For every coframe $K$, $K_\Box$ is a frame whose binary meets are given by the formula
    \begin{equation}\label{152eq}\tag{$\ast$}
        a\stackrel{K_\Box}{\wedge}b=\tbigvee_{i, j} (a_i\wedge b_j),
    \end{equation}
    where $a, b \in K_\Box$, $a=\tbigvee_i a_i$, $b=\tbigvee_j b_j$,  with $a_i, b_j \in \mc$.
\end{prop}
\begin{proof}
It is obvious that $K_\Box$ is complete (with joins computed as in $K$). Moreover, because $\mc$ is closed under meets, we have
$$\tbigvee_{i,j}(a_i \wedge b_j)\leq a\stackrel{K_\Box}{\wedge} b .$$

For the other inequality, consider $c\in \mc$ such that $c\leq a, b$. By the previous lemma,
$$c=c\wedge a= c \wedge \tbigvee_i a_i= \tbigvee_i (c \wedge a_i).$$
Furthermore, we can apply the lemma again to deduce
$$c= \tbigvee_j (b_j \wedge c)= \tbigvee_j (b_j \wedge \tbigvee_i (c \wedge a_i))= \tbigvee_{i,j} (b_j \wedge a_i \wedge c )\leq \tbigvee_{i,j}(a_i \wedge b_j).$$
Since this holds for every $c\in C$ such that $c\le a, b$ and $\mc$ is a join-base of $K_\Box$, it follows that $$a \stackrel{K_\Box}{\wedge} b\leq\tbigvee_{i,j}(a_i \wedge b_j).$$

To see that $K_\Box$ is a coframe, take a family $d_j \in K_\Box$  for $ j \in J^\prime$ and for every $j\in J^\prime$ let 
$$Y_j=\{c \in \mc \mid c\le d_j\} \qtq{and} Y=\tbigcup_jY_j.$$
Clearly $d_j = \tbigvee Y_j$ and therefore $\tbigvee_j d_j=\tbigvee Y$. Moreover, by the formula (\ref{152eq}), we have
$$a \stackrel{K_\Box}{\wedge} \tbigvee_j d_j = (\tbigvee_i a_i) \stackrel{K_\Box}{\wedge}  (\tbigvee_{c\in Y} c~) = \tbigvee_{c\in Y,i\in I} (c\wedge a_i ) \leq \tbigvee_{i,j}(a_i \wedge d_j)\leq\tbigvee_j (a\wedge d_j).\qedhere$$
\end{proof}

\section{Separation conditions generalizing subfitness}\label{sec lifted subfit}
\subsection{$\mo$-Lifted-Subfitness}
As we have seen, the subfitness condition on a locale may be characterized by the behavior of  the coframe of its sublocales. In fact, by Corollary \ref{coro subfit locales} a locale $L$ is subfit if and only if
\begin{equation}\label{lsf}
    S^\sim=L\Leftrightarrow S=L \quad\mbox{for all } S\in \SL(L).
\end{equation}
The following definition extends condition (\ref{lsf}) from the coframe $\SL(L)$ to every coframe $K$ and every $\mo$-closure.

\begin{defin}\label{defin subfit}
A coframe $K$ is said to be $\mo$-lifted-subfit if
\begin{equation*}
     s^\sim=1\Leftrightarrow s=1 \quad \mbox{for every }s \in K.
\end{equation*}
\end{defin}
Whenever the family $\mo$ is clear from the context, we will simply say that $K$ is lifted-subfit, instead of $\mo$-lifted-subfit.

\begin{rema}It is important to keep in mind that a locale $L$ is subfit if and only if $\SL(L)$ is $\of L$-lifted-subfit. 
\[\begin{tikzcd}[ row sep=1cm, column sep=1cm]
\SL(L)&   \of L\textup{-lifted-subfit}\\
 L   \arrow[u, smalltext=lift]& \textup{subfit} \arrow[u, Leftrightarrow ]
\end{tikzcd}\]
\end{rema}
Observe that this approach allows us to unify notions from point-free and classical topology.
\begin{prop}\label{t1 subfit}
    Let $(X,\tau)$ be a topological space. The coframe $\mathcal{P}(X)$ is $\tau$-lifted-subfit if and only if $(X, \tau)$ is $T_1$.   
\end{prop}

All the results we will prove about lifted-subfit coframes are consequences of the following simple observation:
 \begin{lemm}\label{lemasubfit}
     A coframe $K$ is lifted-subfit if and only if for every $s, t \in K$
     $$s^\sim\vee t=1 \Leftrightarrow s \vee t=1.$$
 \end{lemm}
\begin{proof}
Suppose that $K$ is lifted-subfit. Then, by Proposition \ref{prop sim distr}
 $$s^\sim\vee t=1 \Leftrightarrow( s^\sim\vee t)^\sim=1 \Leftrightarrow (s \vee t)^\sim=1\Leftrightarrow s \vee t=1.$$
 To prove the converse, just take $t=0$.
\end{proof}

We can now obtain an apparently stronger characterization of lifted-subfit coframes.
\begin{coro}\label{coro subfit in the opens}
    A coframe $K$ is lifted-subfit if and only if, for all $s\in K$ and $a\in \mo$
    $$s^\sim =a\Leftrightarrow s=a .$$
\end{coro}
\begin{proof}
Suppose that $K$ is lifted-subfit. It suffices to check that,  for all $s\in K$ and $a\in \mo$
$s^\sim =a\Rightarrow s\geq a .$
To do so, use Lemma \ref{lemasubfit} and the fact that $a$ is complemented:
$$s^\sim =a\Rightarrow s^\sim \vee  a^\#=1 \Rightarrow s \vee  a^\# =1 \Rightarrow  s\geq a .$$

For the converse, take $a=1$.
\end{proof}

We now characterize lifted-subfitness by some equalities on the operators $\sim, \Box$ and $\#$.
\begin{coro}\label{coro subfit equalities}
    The following conditions on a coframe $K$ are equivalent:
    \begin{enumerate}
        \item $K$ is lifted-subfit.
        \item For every $s\in K$, $s^{\sim \#}=s^\#$ . 
        \item For every $s\in K$, $s^{\sim \Box}=s^\Box$.
    \end{enumerate}
\end{coro}
\begin{proof}
    (1)$\Rightarrow$(2): Applying Lemma \ref{lemasubfit} we get
    $$s^{\sim \#}=\tbigwedge\{t\in K \mid s^\sim \vee t=1\}=\tbigwedge\{t\in K \mid s \vee t=1\}= s^\#.$$

    $(2)\Rightarrow(3)$: Using Proposition \ref{boxsim} we obtain,
    $s^{\sim\Box}=(s^{\sim\#})^{\sim\#}= s^{\#\sim\#}=s^\Box.$
    
      $(3)\Rightarrow(1)$: Observe that
    $$s^\sim =1 \Leftrightarrow s^{\sim \Box}=1\Leftrightarrow s^\Box =1\Leftrightarrow s=1.\qedhere $$
\end{proof}

\begin{coro}\label{coro subfit suplements}
    Consider the following conditions on a coframe $K$:
    \begin{enumerate}
        \item $K$ is lifted-subfit.
        \item For every $s\in K$, $s^\Box=s^{\#\#}$.
        \item $K_\Box=K_\#$ (i.e. $K_\Box$ is the Booleanization of $K$).
        \item For every $a\in \mo$, $a^\Box= a$.
    \end{enumerate}
    Then $(1)\Rightarrow(2)\Leftrightarrow(3)\Rightarrow (4)$.
\end{coro}
\begin{proof}
    $(1)\Rightarrow(2)$: By Proposition \ref{boxsim} and Corollary \ref{coro subfit equalities} (2) it follows that, for every $s\in K$, 
    $s^\Box=(s^\#)^{\sim \#}=s^{\#\#}.$
    
    $(2)\Rightarrow(3)$: Just observe that
$K_\#=\{s^{\#\#}\mid s\in K\}=\{s^{\Box}\mid s\in K\}=K_\Box.$

    $(3)\Rightarrow(2)$: Let $s\in K$. Note that $s^{\#\#}\in K_\#=K_\Box$, hence $s^{\#\#}=s^{\#\#\Box}$. Now, applying Corollary \ref{sup} we conclude that
    $s^{\#\#}=s^{\#\#\Box}=s^\Box.$
    
    $(2)\Rightarrow(4)$: Since every $a\in \mo$ has a complement it follows that
    $a^\Box= a^{\#\#}=a.$
\end{proof}

\begin{remas}\label{remark ex}
   (I) In general, condition (2) does not imply condition (1). For example take
    $K=\{0< a< 1\},$
    and consider $\mo=\{0, 1\}.$

   (II) We also point out that condition (4) does not imply condition (2). For example, take $B$ to be the following Boolean algebra
    
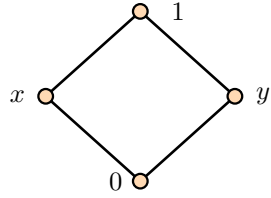
\begin{center}
	\begin{tikzpicture}[scale=.5]
		\draw[line width=0.35mm] (2.5,0) -- (0,2.25);
		\draw[line width=0.35mm] (2.5,0) -- (5,2.25);
		\draw[line width=0.35mm] (2.5,4.5) -- (0,2.25);
		\draw[line width=0.35mm] (2.5,4.5) -- (5,2.25);
		\node at (1.85,0) {\small$0$};
		\node at (-.75,2.25) {\small$x$};
		\node at (5.75,2.25) {\small$y$};
		\node at (3.5,4.5) {\small$1$};
		\draw[fill=orange!30!white,line width=0.30mm] (2.5,0) circle (5.5pt);
		\draw[fill=orange!30!white,line width=0.30mm] (0,2.25) circle (5.5pt);
		\draw[fill=orange!30!white,line width=0.30mm] (5,2.25) circle (5.5pt);
		\draw[fill=orange!30!white,line width=0.30mm] (2.5,4.5) circle (5.5pt);
	\end{tikzpicture}
\end{center}
and take $\mo=\{0, 1\}$. 
It is clear that condition (4) holds, nevertheless $x^\Box=0$ but $x^{\#\#}=x$.
\end{remas}

Applying the previous result when $L$ is a locale, $K=\SL(L)$ and $\mo=\of L$ and using Corollary \ref{coro subfit locales} it immediately follows that:

\begin{coro}
    The following statements about a locale $L$ are equivalent:
    \begin{enumerate}
        \item $L$ is subfit.
        \item For every $S\in \SL(L)$, $S^\Box=S^{\#\#}$.
        \item $\SL_\cf(L)$  is the Booleanization of $\SL(L)$.
        \item For every $a\in L$, $\of a^\Box=\of a$.
    \end{enumerate}
\end{coro}

\subsection{Quasi lifted-subfitness: A weaker version of lifted-subfitness}

We now look at an alternative characterization of subfitness of a locale $L$ in terms of the properties of its sublocales.

As we have seen in Corollary \ref{coro subfit locales}, a locale $L$ is subfit if and only if
\begin{equation}\label{q-l-sf}
    \of a^\Box=\of a \quad\mbox{ for all } a \in L.
\end{equation}
The following definition extends condition (\ref{q-l-sf}) from the coframe $\SL(L)$ to every coframe $K$ and every $\mc$-interior.
\begin{defin}
    We say that a coframe $K$ is quasi $\mo$-lifted-subfit if, 
    $a^\Box=a \quad\mbox{for all } a\in \mo.$
\end{defin}
    Whenever $\mo$ is clear from the context, we will simply say that $K$ is quasi lifted-subfit, instead of quasi $\mo$-lifted-subfit.

As an immediate consequence of Corollary \ref{coro subfit suplements} we obtain:
\begin{coro}\label{coro ls implies qlf}
    A lifted-subfit coframe is quasi lifted-subfit.
\end{coro}

\begin{rema}\label{rema ex quasi subfit}
Contrary to what happens in the coframe of sublocales (with $\mo=\of L$), a general coframe may be quasi lifted-subfit without being lifted-subfit. For an example, take $B$ and $\mo$ as defined in Remark \ref{remark ex}.

In section \ref{sec 0dim} we will see a sufficient condition to guarantee that a quasi lifted-subfit coframe is lifted-subfit.
\end{rema}  

It is an easy exercise to prove that:
\begin{prop}
    A coframe $K$ is quasi lifted-subfit if and only if
    $$a\leq s \Leftrightarrow a\leq s^\Box, \quad \mbox{for all } s\in K \text{ and } a\in \mo.$$
\end{prop}

\subsection{A even weaker version of lifted-subfitness}

As we have seen, given a locale $L$, $\SL (L)$ is $\of L$-lifted-subfit if and only if $L$ is subfit. Since $L$ is isomorphic to $\of L$ we immediately obtain that
 $\SL (L)$ is $\of L$-lifted-subfit if and only if $\of L$ is subfit.
 
In this section we discuss how the $\mo$-lifted-subfitness of $K$ is related to the subfitness of $\mo$ (viewed as a subposet of $K$). To do so, start by observing that the definition of a subfit locale only depends on the existence of binary joins and a top. Hence, we can directly extend the definition of subfit locale to every join-semilattice with a top element.  

\begin{defin}
        A join-semilattice with top $A$ is said to be \emph{subfit} when
    \begin{center}
        if $a\not \leq b$, then there exists $c\in A$ such that $a\vee c=1\not =b\vee c.$
    \end{center}
\end{defin}

Since $\mo$ contains $1$ and is closed under finite joins in $K$, it is a join-semilattice with top (with respect to the order inherited from $K$). Hence, we can ask the question
    ``When is $\mo$ a subfit join-semilattice with top?"

It is easy to show that
\begin{prop}\label{prop subfit in mo}
The following statements are equivalent
    \begin{enumerate}
        \item For all $a, b\in \mo$
        $$a^\Box\le b^\Box\Rightarrow a\le b.$$
        \item Let $a, b \in \mo$. If for all $c\in \mc$
        $$c\le a \Rightarrow c\le b,$$
        then $a\le b$.     
        \item Let $a, b \in \mo$. If for all $c\in \mc$
        $$c^\#\vee a=1 \Rightarrow c^\#\vee b=1,$$
        then $a\le b$.
        \item $\mo$ is subfit.
    \end{enumerate}
\end{prop}
\begin{coro}\label{coro qsf implies mosf}
    Let $K$ be a coframe. If $K$ is quasi lifted-subfit, then $\mo$ is subfit.
\end{coro}

In Section \ref{sec 0dim} we will see a sufficient condition to guarantee that a coframe is quasi $\mo$-lifted-subfit if and only if $\mo$ is subfit.

\section{Consequences of lifted-subfitness}\label{sec consequencias lifte subfit}

\subsection{Supplements in $\SL(L)$}

As an application of our results, we now present a formula for the computation of supplements in $\SL(L)$ whenever $L$ is a subfit frame. To do so, we have to recall the definition of a nucleus.

\begin{defin}
 Let $L$ be a locale. A \emph{nucleus} on $L$ is a closure operator $\nu$ on $L$ satisfying the property
$$\nu(a\wedge b)=\nu (a)\wedge\nu (b)\quad\mbox{ for all } a,b\in L.$$
\end{defin}

Nucleus and sublocales are related by the following proposition.
\begin{theo}[\cite{PP}]\label{sublocales nucleus}
    Let $L$ be a locale and $\nu \colon L\to L$ a nucleus. Then $\nu [L]$, the set-theoretic image of $L$ along $\nu$, is a sublocale. Vice versa, given a sublocale $S\subseteq L$, the map
        \begin{equation*} 
		\xymatrix@C=25pt@R=10pt{
			\nu_S\colon L  \ar[r]^{}  & L \\
			\hspace*{12mm}x\ar@{|->}[r]_{} & \tbigwedge \{s\in S \mid x\le s\}}
	\end{equation*}
    is a nucleus. Moreover $\nu_S[L]=S$ and $\nu_{\nu[L]}=\nu$.
\end{theo}
We will need the following property of nuclei.

\begin{lemm}[\cite{PP}, pp. 74]
    Let $L$ be a locale and $S$ a sublocale of $L$. Then
    $$S\sue \of a \Leftrightarrow \nu_S(a)=1.$$
\end{lemm}

\begin{prop}\label{prop sup em sl}
    A locale $L$ is subfit if and only if
    $$S^\#=\{\tbigwedge K \mid K\sue \nu_S^{-1}[\OS] \} \quad \mbox{for every } S\in \SL(L).$$
\end{prop}
\begin{proof}
    By \ref{coro subfit equalities}, it suffices to prove that, for every $S\in \SL(L)$,
    $$S^{\sim \#}=\{\tbigwedge K \mid K\sue \nu_S^{-1}[\OS] \}.$$
    By the previous lemma, we have
    $$S^{\sim \#}=(\tbigcap\{\of a\mid S\sue \of a\})^\#=\tbigvee \{\cf a \mid a\in \nu_S^{-1}[\OS]\}.$$
    Now let $A= \tbigcup \{\cf a\mid  a\in \nu_S^{-1}[\OS]\}$. Clearly
    $S^{\sim \#}=\{\tbigwedge K \mid K\sue A\}.$
    Hence, it suffices to show that $A=\nu_S^{-1}[\OS]$.

    Clearly $\nu_S^{-1}[\OS]\sue A$. On the other hand, if $x\in A$ there exists some $a\in \nu_S^{-1}[\OS]$ such that $a\leq x$, hence  $1=\nu_S(a)\le \nu_S(x)$ and therefore $x\in \nu_S^{-1}[\OS]$.
\end{proof}

\begin{coro}\label{coro nu}
    A locale $L$ is subfit if and only if
    $$\nu_{S^\#}(x)=\tbigwedge\{ k \in \nu_S^{-1}[\OS]\mid x\le k \}$$
    for all  $S\in \SL(L)$  and all  $x\in L$.
\end{coro}

The particular case $S=\cf a$ yields the well known formula

\begin{coro}
    In any subfit locale $L$,
    $$a\to x =\tbigwedge\{k\in L \mid a\vee k=1 \mbox{ and } x\le k\},$$
    for every $a, x\in L$.
\end{coro}

Using the Proposition \ref{prop sup em sl} and the well-known fact  \cite{remainders} that, for every $S\in \SL(L)$,
$$S^\circ=\of (\tbigwedge S^\#),$$
 it immediately follows that:
\begin{coro}
    Let $L$ be a subfit  locale and $S$ a sublocale of $L$. Then 
    $$S^\circ=\of (\tbigwedge \nu_s^{-1}[\OS]).$$
\end{coro}

\subsection{A result on adjunctions}
Recall that a localic map $f\colon L\to M$ is said to be open if 
$f[\of a]\in \of M$ for all $ a\in L.$

Open maps between locales are characterized by the celebrated \emph{Joyal-Tierney's Theorem} \cite{JT84}:

\begin{quote}
{\em
A localic map $f \colon L \to M$ is open if and only if its adjoint frame homomorphism
$f^*\colon M\to L$  has itself a left adjoint and preserves the Heyting operation. }
\end{quote}
In the case where $M$ is a subfit locale, we have a more elegant characterization of open maps \cite{PPsep}.

\begin{theo}\label{open subfit}
    Let $L$ and $M$ be locales, with $M$ subfit. Then a localic map $f \colon L \to M$ is open if and only if its adjoint frame homomorphism
$f^*\colon M\to L$ has itself a left adjoint.
\end{theo}
Our goal in this section is to generalize this result. We will do it by studying the behavior of adjunctions between a poset and a lifted-subfit coframe. We will then be able to derive the previous theorem by applying our result to the adjunction between the set-theoretical image of $f$ and the localic preimage of $f$ (which happens between the poset $\SL(L)$ and the $\of L$-lifted-subfit coframe $\SL(M)$).

\medskip
\begin{prop}\label{theo adj}
    Let $P$ be a poset and $K$ be a coframe. Let also $g\colon P\to K$ and $h\colon K\to P$ be such that $g\dashv h$. Consider further any $\mo_P\sue P$ and $\mo_K\sue K$ such that $K$ is $\mo_K$-lifted-subfit and $h[\mo_K]\sue \mo_P$. Then the map
          \begin{equation*}
		\xymatrix@C=45pt@R=10pt{
			h_\mo\colon\mo_K \ar[r]^{}  &\mo_P \\
			\hspace*{10mm} b \ar@{|->}[r]_{} & h(b)}
	\end{equation*}
    
    has a left adjoint $h_\mo^\ast$ if and only if $g[\mo_P]\sue \mo_K$. In that case 
    $$h_\mo^\ast(a)=g(a) \mbox{ for all } a\in \mo_K.$$
\end{prop}
\begin{proof}
    Suppose $h_\mo$ has a left adjoint $h_\mo^\ast$. Then for every $a, b \in \mo_K$
    $$h_\mo^\ast(a) \le b \Leftrightarrow a \le h(b) \Leftrightarrow g(a)\le b.$$
    Since $h_\mo^\ast(a)\in \mo_K$ this implies that $g(a)^\sim=h_\mo^\ast(a)^\sim=h_\mo^\ast(a)$.  Because $K$ is $\mo_K$-lifted-subfit and $h_\mo^\ast(a)\in \mo_K$, we can apply Corollary \ref{coro subfit in the opens} to conclude that
    $$g(a)=h_\mo^\ast(a) \in \mo_K \quad \mbox{ for all } a \in \mo_P.$$
    The other implication is trivial.
\end{proof}
To see how the previous proposition generalizes the already known result for the case of locales, we will use it to derive Theorem \ref{open subfit}. We will need the following technical lemma.

\begin{lemm}
    Let $f\colon L\to M$ be a localic map between locales. Then the map
          \begin{equation*}
		\xymatrix@C=45pt@R=10pt{
			f_{-1}^\of[-] \colon\of M \ar[r]^{}  &\of L \\
			\hspace*{10mm} \of b \ar@{|->}[r]_{} & f_{-1}[\of b]}
	\end{equation*}
    has a left adjoint $h\colon\of L\to \of  M$  if and only if $f^\ast \colon M \to L$ has a left adjoint $f_!\colon L\to  M$, and in that case
    $$h[\of a]=\of (f_!(a))\quad \mbox{for all }a \in L.$$
\end{lemm}
\begin{proof}
      Since $f$ is a localic map  $f_{-1}^\of[\of b]=\of (f^\ast (b))$. Moreover, since the map 
\begin{equation*}
		\xymatrix@C=45pt@R=10pt{
			\of \colon L  \ar[r]^{}  & \of L \\
			\hspace*{10mm} a \ar@{|->}[r]_{} & \of a}
	\end{equation*}
    is an isomorphism, there exists a one-to-one correspondence between maps $h\colon\of L\to \of  M$ and maps $f_!\colon L\to  M$, such that for each $h$, the associated map $f_!$ makes the following diagram commute:
\begin{center}
\begin{tikzcd}[column sep=large, row sep=large]
\of L \arrow[r] \arrow[r, "h"] & \of  M \arrow[r] \arrow[r, "f_{-1}^\of"] & \of L \\
L \arrow[u] \arrow[u, "\of"] \arrow[r, "f_!"]  &   M \arrow[u] \arrow[u, "\of"] \arrow[r, "f^\ast"] &  L \arrow [u] \arrow[u, "\of"].
\end{tikzcd}
\end{center}
Now, we use the fact that $\of\colon L\to \of L$ is an isomorphism and the commutativity of the previous diagram to deduce that $f_{-1}^\of$ has a left adjoint $h\colon\of L\to \of  M$  if and only if $f^\ast$  has a left adjoint $f_!\colon L\to  M$.
\end{proof}

\begin{theo}
    Let $L$ and $M$ be locales, with $M$ subfit. Then a localic map $f \colon L \to M$ is open if and only if its adjoint frame homomorphism
$f^*\colon M\to L$ has itself a left adjoint.
\end{theo}
\begin{proof}
Observe that we can apply Proposition \ref{theo adj} to the case where $P=\SL(L)$, $K=\SL(M)$, $\mo_P=\of L$,  $\mo_K=\of M$  and the maps $g$ and $h$ there, here become the set-theoretical image of $f$ and the localic preimage of $f$.

    Hence, we can conclude that the map
              \begin{equation*}
		\xymatrix@C=45pt@R=10pt{
			f_{-1}^\of[-] \colon\of M \ar[r]^{}  &\of L \\
			\hspace*{10mm} \of b \ar@{|->}[r]_{} & f_{-1}[\of b]}
	\end{equation*}
    has a left adjoint if and only if $f[\of L]\sue \of M$. The result then follows from the previous lemma.
\end{proof}

\section{Separation conditions generalizing fitness}\label{sec lifted fit}
\subsection{$\mo$-lifted-fitness}
As we have seen, the fitness condition on a locale may be characterized by the behavior of the coframe of its sublocales. In fact, by 
 Corollary \ref{coro fit locales} a locale $L$ is fit if and only if
 \begin{equation}\label{fit}
S^\sim=S \quad\mbox{ for all } S \in \SL(L).
 \end{equation}
The following definition extends condition (\ref{fit}) from the coframe $\SL(L)$ to every coframe $K$ and every $\mo$-closure.
\begin{defin}
    We say that a coframe $K$ is $\mo$-lifted-fit if
    $$s^\sim=s \quad \mbox{ for all } s\in K.$$
\end{defin}

 Whenever clear we are dealing with the $\mo$ fixed in Section \ref{sec oclo cint}, we will simply write $K$ is lifted-fit, instead of $\mo$-lifted-fit.

\begin{rema}It is important to keep in mind that a locale $L$ is fit if and only if $\SL(L)$ is $\of L$-lifted-fit. 
\[\begin{tikzcd}[ row sep=1cm, column sep=1cm]
\SL(L)&   \of L\textup{-lifted-fit}\\
 L   \arrow[u, smalltext=lift]&  \textup{fit} \arrow[u, Leftrightarrow ]
\end{tikzcd}\]
\end{rema}
This approach allows us to unify notions from point-free and classical topology.
\begin{prop}\label{prop fit top}
    Let $(X,\tau)$ be a topological space, the coframe $\mathcal{P}(X)$ is $\tau$-lifted-fit if and only if $(X, \tau)$ is $T_1$.   
\end{prop}
\begin{prop}
    The following statements about a coframe $K$ are equivalent:
    \begin{enumerate}
        \item $K$ is lifted-fit.
        \item Let $s, t \in K$. If for any  $a\in \mo$
        $s\leq a \Rightarrow t\leq a,$
        then $t\leq s$.
    \end{enumerate}
\end{prop}
\begin{proof}
    Suppose $K$ is lifted-fit and  $s, t \in K$. If 
        $s\leq a \Rightarrow t\leq a$ for every  $a\in \mo$,
        then $t^\sim \leq s^\sim$, and therefore $t\leq s$.

    On the other hand, observe that
        $s\leq a \Rightarrow s^\sim\leq a.$
    Therefore $s^\sim \leq s$.
\end{proof}

\subsection{Quasi lifted-fitness: A weaker version of lifted-fitness}
We now look at an alternative characterization of fitness of a locale $L$ in terms of the properties of its sublocales.

As we have seen in Corollary \ref{coro fit locales}, a locale $L$ is fit if and only if
\begin{equation}\label{qlf}
\cf a^\sim=\cf a \quad\mbox{ for all } a \in L.
\end{equation}
The following definition extends condition (\ref{qlf}) from the coframe $\SL(L)$ to every coframe $K$ and every $\mo$-closure.
\begin{defin}
    We say that a coframe $K$ is quasi $\mo$-lifted-fit if
    $$c^\sim=c \quad \mbox{for all }c \in \mc.$$
\end{defin}
 Whenever clear we are dealing with the $\mo$ fixed in Section \ref{sec oclo cint},  we will simply write $K$ is quasi lifted-fit, instead of quasi $\mo$-lifted-fit.

We now describe the relation between lifted-fitness, quasi lifted-fitness, and the operators $\sim$, $\Box$, and $\#$. 
\begin{prop}
    Consider the following statements about a coframe $K$:
    \begin{enumerate}
        \item $K$ is lifted-fit.
        \item For every $s\in K$, $s^{\#\sim}=s^{\sim\#}$.
        \item For every $s\in K$, $s^{\#\#\sim}=s^{\Box}$.
        \item $K$ is quasi lifted-fit.
    \end{enumerate}
Then $(1)\Rightarrow (2)\Rightarrow (3)\Rightarrow (4)$.
\end{prop}
\begin{proof}
    $(1)\Rightarrow (2)$: Trivial.

    $(2)\Rightarrow (3)$: By Proposition \ref{boxsim} we have
    $(s^\#)^{\#\sim}=s^{\#\sim\#}=s^{\Box}.$
        
    $(3)\Rightarrow (4)$: Take $s=c\in \mc$, then
    $c^\sim =c^{\#\# \sim}= c ^{\Box} = c.$
\end{proof}

It is an easy exercise to check the following proposition.
\begin{prop}
    A coframe $K$ is quasi lifted-fit if and only if
    $$s^\sim\le c \Leftrightarrow s\leq c \quad \mbox{ for all } c\in \mc \mbox{ and } s\in K.$$
\end{prop}

\subsection{Summarizing}

We will now present the relations between the separation conditions we have been studying.
\begin{lemm}
    Let $K$ be a coframe. If $K$ is quasi lifted-fit, then it is quasi lifted-subfit.
\end{lemm}
\begin{proof}
    Let $a\in \mo$. Since $a$ is complemented, we can apply Corollary \ref{sup} to deduce that
    $$a^\Box=(a^\#)^{\#\Box}=(a ^{\#\sim})^\#=a^{\#\#}=a.\qedhere$$
\end{proof}
The following diagram shows the relations between all the separation conditions previously studied, where an arrow $\to$ stands for implies.
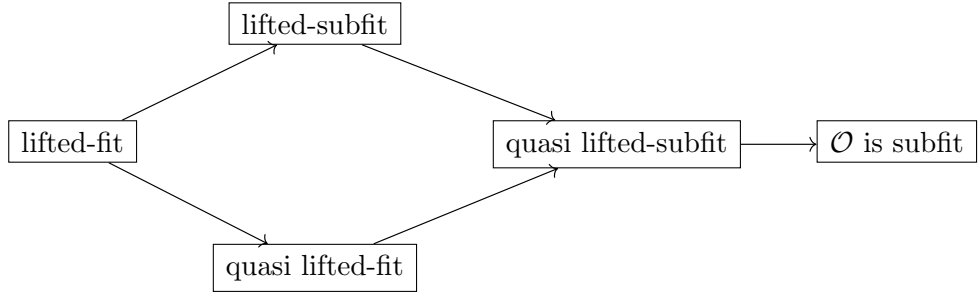
\[\begin{tikzcd}[cells={nodes={draw}}, row sep=1cm, column sep=1cm]
& \textup{lifted-subfit} \arrow[dr, ] & &\\
\textup{lifted-fit}     \arrow[ur, ]     \arrow[dr,]& & \textup{quasi lifted-subfit} \arrow[r, ]
    &\mo \textup{ is subfit}\\
& \textup{quasi lifted-fit}\arrow[ru, ]& &
\end{tikzcd}\]

\subsection{``Dualizing" lifted-fitness}

A natural question that arises when dealing with lifted-fitness is ``What if, rather than demanding that 
$$s^\sim=s \quad \mbox{for all } s \in K,$$
we demand instead that
$$s^\Box=s \quad \mbox{for all } s \in K?\mbox{"}$$
It turns out that this ``dual version of lifted-fitness" is far stronger than  lifted-fitness; in fact, we have:
\begin{prop}\label{box id}
    The following conditions on a coframe $K$ are equivalent:
    \begin{enumerate}
        \item For every $s\in K$, $s^\Box=s$.
        \item  $K$ is lifted-fit and Boolean.
        \item  $K$ is lifted-subfit and Boolean.
    \end{enumerate}
\end{prop}

\begin{proof}
    $(1)\Rightarrow (2)$: Take $s\in K$. By Corollary \ref{sup}, we have that
    $s=s^\square=s^{\#\# \square}=s^{\#\#}.$
   
   Since this holds for every $s\in K$, $K$ must be Boolean. Because $s$ is complemented it follows from Proposition \ref{boxsim} that
    $$s^\#=(s^\#)^\square=s^{\#\#\sim \#}=s^{\sim\#}.$$
    Since $s$ and $s^\sim$ are both complemented, this implies that $s=s^\sim$.

    $(2)\Rightarrow (3)$: It is trivial.

    $(3)\Rightarrow (1)$: By Corollary \ref{coro subfit suplements} we obtain
    $s^\square=s^{\#\#}=s.$ 
\end{proof}

\begin{rema}
    Recall that by Propositions \ref{t1 subfit} and \ref{prop fit top} we know that, given a topological space $(X, \, \tau)$, $\mathcal{P}(X)$ is $\tau$-lifted-subfit if and only if $\mathcal{P}(X)$ is $\tau$-lifted-fit. Since $\mathcal{P}(X)$ is  Boolean,  this equivalence is a particular case of the previous proposition.
\end{rema}

\section{The heredity of the separation condition}\label{sec heredi}
It is well-known that the separation conditions on a locale $L$ tend to be characterized by the separation conditions on its sublocales. In fact:
\begin{prop}[\cite{PPsep}]\label{prop sep sf in sublocales}
    A locale $L$ is subfit if and only if every one of its \emph{open} sublocales is subfit.
\end{prop}
As for fitness, we have:
\begin{prop}[\cite{PPsep}]\label{prop sep f in sublocales}
    A locale $L$ is fit if and only if every one of its sublocales is subfit.
\end{prop}
The goal of this section is to generalize these results to our more abstract context. To do so, we start by recalling some basic facts from \cite{PP}.
\begin{prop}
    Let $L$ be a locale and $S\sue L$ a sublocale of $L$. Then $S$ is still a locale and 
    $$\SL(S)=\{T\in \SL(L)\mid T\sue S \}=\downarrow S.$$
    Moreover
    $$\of S=\{\of a \cap S\mid \of a \in \of L\} \qtq{and} \cf S=\{\cf a \cap S\mid \cf a \in \cf L\}.$$
\end{prop}
This motivates us to study how the separation conditions on a coframe $K$ can be characterized by the separation condition on its principal down-sets. 

\bigskip

For every $s\in K$ set
$$\mo_s=\{a\wedge s\mid a \in \mo\}\quad \mbox{ and } \quad \mc_s=\{c\wedge s\mid c \in \mc\}.$$
For each $t\in\, \downarrow s$ we define
\begin{itemize}
    \item the $\mc$-interior of $t$ relative to $s$: $t^{\Box_s}=\tbigvee \{c\in \mc_s \mid c\leq t\};$
    \item the $\mo$-closure of $t$  relative to $s$: $t^{\sim_s}=\tbigwedge \{a\in \mo_s \mid t\leq a\}.$
\end{itemize}
\begin{lemm}
    For every $ s\in K$ and $t\in \,\downarrow s$
    $$t^{\sim_s}=s\wedge t^\sim.$$
In particular, if $s\in K_\sim$
    $$t^{\sim_s}=t^\sim.$$
\end{lemm}
\begin{proof}
    Observe that 
    \begin{align*}
        t^{\sim_s}&=\tbigwedge \{a\in \mo_s \mid t\leq a\}
        =\tbigwedge \{a\wedge s \mid t\leq a\wedge s, \; a\in \mo\}
        =s\wedge \tbigwedge \{a\in \mo \mid t\leq a\}
        = s\wedge t^\sim.
    \end{align*}
Suppose now that $s\in K_\sim$. Since $t\le s$, we have $t^\sim\le s^\sim=s$ and therefore
    $$t^{\sim_s}=s\wedge t^\sim=t^\sim.\qedhere$$
\end{proof}

\begin{remas}
    (1): Observe that, by definition, given $s\in K$, $\downarrow s$ is $\mo_s$-lifted-subfit if and only if
    $$t^{\sim_s}=s \Leftrightarrow t=s \quad \mbox{for all } t \in \,\downarrow s.$$
    (2): Observe that, by definition, given $s\in K$, $\downarrow s$ is $\mo_s$-lifted-fit if and only if
    $$t^{\sim_s}=t \quad \mbox{for all } t \in \,\downarrow s.$$
\end{remas}
We can now generalize Proposition \ref{prop sep sf in sublocales} in the following way:
\begin{prop}
    A coframe $K$ is $\mo$-lifted-subfit if and only if $\downarrow a$ is $\mo_a$-lifted-subfit for all $a\in \mo$.
\end{prop}
\begin{proof}
    
Let $K$ be $\mo$-lifted-subfit and take $a\in \mo$ and $t\in\, \downarrow a$. Since $a\in K_\sim$, $t^{\sim_a}=t^\sim$, hence we can use Corollary \ref{coro subfit in the opens} to deduce that
$$t^{\sim_a}=a\Leftrightarrow t^\sim = a\Leftrightarrow t=a.$$
Hence $\downarrow a$ is $\mo_a$-lifted-subfit.

For the converse, just take $a=1$.
\end{proof}

We now prove some similar results for lifted-fit coframes.
\begin{prop}
     A coframe $K$ is $\mo$-lifted-fit if and only if $\downarrow s$ is $\mo_s$-lifted-fit for all $s\in K$.
\end{prop}
\begin{proof}
    Let $K$ be $\mo$-lifted-fit and take $s\in K$ and $t\in\, \downarrow s$, then
    $t^{\sim_s}=s\wedge t^\sim=s\wedge t= t.$
    Hence $\downarrow s$ is $\mo_s$-lifted-fit.

    For the converse, just take $s=1$.
\end{proof}

\begin{prop}
     A coframe $K$ is $\mo$-lifted-fit if and only if $\downarrow s$ is $\mo_s$-lifted-subfit for all $s\in K$.
\end{prop}
\begin{proof}
    
$\Rightarrow:$ Follows immediately from the previous proposition.

$\Leftarrow$: Let $t\in K$ and set $s=t^\sim$. Since $s\in K_\sim$ and $\downarrow s$ is $\mo_s$-subfit it follows that
$$t^\sim=s\Rightarrow t^{\sim_s}=s\Rightarrow t=s \Rightarrow t =t^\sim.$$
Hence $K$ is fit.
\end{proof}

\section{Separation in $\mo$-zero-dimensional coframes}\label{sec 0dim}
By Corollary \ref{coro subfit locales} we know that, if $L$ is a locale, then
\begin{center}
$\SL(L)$ is $\of L$-lifted-subfit if and only if $\SL(L)$ is quasi $\of L$-lifted-subfit.
\end{center}
This does not happen in general. In fact, a coframe can be quasi lifted-subfit but not lifted-subfit (see Remark \ref{rema ex quasi subfit}). This raises the question:
\begin{center}
    ``What is the property of $\SL(L)$ that guarantees the equivalence between $\of L$-lifted-subfitness and quasi $\of L$-lifted-subfitness?"
\end{center}
The main goal of this section is to answer that question. To do so, we recall the following fact about sublocales:
\begin{prop}\label{prop zero dim of sl}
    Let $L$ be a locale and $S\sue L$ a sublocale of $L$. Then
    $$S=\tbigcap\{\of a \vee \cf b \mid S\sue \of a \vee \cf b\}.$$
\end{prop}
 This property of $\SL(L)$ motivates the following definition.
\begin{defin}
    We say that a coframe $K$ is $\mo$-zero-dimensional if for every $s\in K$ there exists $a_i, b_i\in \mo$ such that
    $$s=\tbigwedge_{i\in I}( a_i \vee b_i^\#).$$
\end{defin}
It is clear that for every locale $L$, $\SL (L)$ is $\of L$-zero-dimensional.

The following results concern the study of the separation conditions on $\mo$-zero-dimensional coframe.
\begin{prop}\label{prop subfit zero sim}
Let $K$ be a $\mo$-zero-dimensional coframe. Then the following statements are equivalent:
\begin{enumerate}
    \item $K$ is $\mo$-lifted-subfit.
    \item  For all $a, b \in \mo$ we have
    $$a\vee b^{\#\sim}=1 \Rightarrow a\vee b^\#= 1.$$
    \item  For all $a, b \in \mo$ we have
    $$ b^{\Box}\le a \Rightarrow  b\le a.$$
\end{enumerate}
\end{prop}
\begin{proof}
    $(1)\Rightarrow(2)$: Follows immediately from Lemma \ref{lemasubfit}.

    $(2)\Rightarrow(1)$: Let $s\in K$ and take $a_i, b_i\in \mo$ such that
    $s=\tbigwedge_{i\in I} (a_i \vee b_i^\#).$
    
    Start by observing that
    $$s^\sim =\left(\tbigwedge_{i\in I} a_i \vee b_i^\#\right)^\sim\le \tbigwedge_{i\in I} (a_i \vee b_i^\#)^\sim. $$
    Therefore, it follows that
    \begin{align*}
        s^\sim=1& \Rightarrow \tbigwedge_{i\in I} (a_i \vee b_i^\#)^\sim =1\\
                & \Rightarrow a_i \vee b_i^{\#\sim} =1\,\text{ for all } i\in I\\
                & \Rightarrow a_i \vee b_i^\# =1\, \text{ for all } i\in I.
    \end{align*}
    Hence
 $s=\tbigwedge_{i\in I} (a_i \vee b_i^\#)=1.$
 
 $(2)\Leftrightarrow(3)$: Follows immediately from the equality $b^{\#\sim \#}=b^\Box$ and the fact that $b$ is complemeted.
\end{proof}

It now follows as a corollary that:
\begin{coro} Let $K$ be a $\mo$-zero-dimensional coframe. Then the following statements are equivalent:
\begin{enumerate}
    \item $K$ is $\mo$-lifted-subfit.
    \item $K$ is quasi $\mo$-lifted-subfit.
    \item $\mo$ (viewed as subposet of $K$) is subfit.
\end{enumerate}
\end{coro}
\begin{proof}
    $(1)\Rightarrow (2)$: Corollary \ref{coro ls implies qlf}.

    $(2)\Rightarrow (3)$:  Corollary \ref{coro qsf implies mosf}.

    $(3)\Rightarrow (1)$: Follows from Propositions \ref{prop subfit in mo} and \ref{prop subfit zero sim}.
\end{proof}

\begin{rema}
    The $\mo$-zero-dimensionality of $K$ plays a fundamental role here as it allows us to write every element $s\in K$ in terms of elements in $\mo$. This enables us to describe a property about all $s\in K$ (namely $\mo$-lifted-subfitness) in terms of the properties of the elements of $\mo$ (namely $\mo$ being subfit).
\end{rema}
We now prove a similar result for lifted-fit coframes.
\begin{prop}
A $\mo$-zero-dimensional coframe $K$ is $\mo$-lifted-fit if and only if it is quasi $\mo$-lifted-fit.
\end{prop}
\begin{proof}
    
$\Rightarrow$: Trivial.

$\Leftarrow$: Let $s\in K$ and take $a_i, b_i\in \mo$ such that
    $s=\tbigwedge_{i\in I} (a_i \vee b_i^\#).$
    
    Since $b^\#\in \mc$ we can conclude that
   $$s^\sim\leq \tbigwedge_{i\in I}( a_i \vee b_i^\#)^\sim=\tbigwedge_{i\in I} a_i^\sim \vee b_i^{\#\sim}=\tbigwedge_{i\in I} a_i \vee b_i^\#=s. \qedhere$$
\end{proof}


\begin{thebibliography}{99}
\bibitem{C} F. Ciraulo, Kuratowski's problem in constructive topology, {\em Journal of Logic and Analysis} 17 (2025) 1-27.

\bibitem{CPP} M.M. Clementino, J. Picado, A. Pultr, The other closure and complete sublocales, {\em Applied Categorical Structures} 26 (5) (2018) 891--906.

\bibitem{ME} M. Erné, {\em Assemblies of implicative semilattices}, preprint, 2024.

\bibitem{remainders} M.J. Ferreira, J. Picado, S.M. Pinto, Remainders in pointfree topology, {\em Topology and its Applications} 245 (2018) 21--45.

\bibitem{JT84} A. Joyal and M. Tierney,
{\em An Extension of the Galois Theory of Grothendieck},
Memoirs of the American Mathematical Society, vol. 309. AMS, Providence, 1984.

\bibitem{PP} J. Picado and A. Pultr,
        {\em Frames and Locales: topology without points},
    Frontiers in Mathematics 28, Springer, Basel, 2012.

    \bibitem{PPsep} J. Picado and A. Pultr,
        {\em Separation in point-free topology},
    Birkh\"auser/Springer, Cham, (2021).


\bibitem{PPT}   J. Picado, A. Pultr and A.Tozzi,  {\em Joins of closed sublocales}, Houston journal of mathematics, Volume 45, Number 1, 2019, Pages 21–38 
\end{thebibliography}
\end{document}